\documentclass[10pt]{amsart}
\usepackage[mathscr]{euscript}

\usepackage{fullpage,graphicx,amsmath,amsthm,caption,epstopdf,tikz,enumitem,color,subcaption, float,lineno,comment}
\usepackage{hyperref}
\usepackage[capitalise, noabbrev,nameinlink]{cleveref}
\usetikzlibrary{snakes}
\newtheorem{theorem}{Theorem}[section]
\newtheorem{lemma}[theorem]{Lemma}
\newtheorem{corollary}[theorem]{Corollary}
\newtheorem{remark}[theorem]{Remark}
\numberwithin{figure}{section}

\newcommand{\thistheoremname}{}
\newtheorem{genericthm}[theorem]{\thistheoremname}

\newtheorem*{genericthm*}{\thistheoremname}
\newenvironment{namedthm*}[1]
{\renewcommand{\thistheoremname}{#1}%
	\begin{genericthm*}}
	{\end{genericthm*}}
\newtheorem*{genericdef*}{\thistheoremname}
\definecolor{LSUG}{RGB}{253,208,35}
\definecolor{LSUP}{RGB}{70,29,124}
\definecolor{AUBN}{RGB}{3,36,77}
\definecolor{AUBN2}{RGB}{73,110,156}
\definecolor{AUBO2}{RGB}{246, 128, 38}
\definecolor{Green}{rgb}{0.0, 0.5, 0.0}

\newcommand*{\equal}{\!=\!}
\newcommand{\ff}{\text{fan-free}}
\newcommand{\cf}{\text{cycle-free}}
\newcommand{\cc}{\text{cross-crowded}}

\newcommand{\X}{\mathcal{X}}
\newcommand{\indfull}{\text{pairwise independent full~}}

\usetikzlibrary{arrows,scopes}

\title{Unavoidable Induced Subgraphs of Large 2-Connected Graphs}
\date{\today}

\author[S. Allred]{Sarah Allred}
\address{Department of Mathematics\\Louisiana State University\\Baton Rouge, LA 70803, USA}
\email{sallre4@lsu.edu}

\author[G. Ding]{Guoli Ding}
\address{Department of Mathematics\\Louisiana State University\\Baton Rouge, LA 70803, USA}
\email{ding@math.lsu.edu}

\author[B. Oporowski]{Bogdan Oporowski}
\address{Department of Mathematics\\Louisiana State University\\Baton Rouge, LA 70803, USA}
\email{bogdan@math.lsu.edu}


\begin{document}
	
	\begin{abstract}
		Ramsey proved that for every positive integer $n$, every sufficiently large graph contains an induced $K_n$ or $\overline{K_n}$.
		Among the many extensions of Ramsey's Theorem there is an analogue for connected graphs: for every positive integer $n$, every sufficiently large connected graph contains an induced $K_n$, $K_{1,n}$, or $P_n$. 
		In this paper, we establish an analogue for 2-connected graphs.
		In particular, we prove that for every integer exceeding two, every sufficiently large 2-connected graph contains one of the following as an induced subgraph: $K_n$, a subdivision of $K_{2,n}$, a subdivision of $K_{2,n}$ with an edge between the two vertices of degree $n$, and a well-defined structure similar to a ladder.
	\end{abstract}
	
	\maketitle
	
	\section{Introduction}
	
	The terms and symbols that are not defined explicitly in this paper will be understood as defined in Diestel~\cite{diestel}.
	This paper focuses on the induced subgraph relation,
	and so we will often wish to state that a graph $G$ contains an induced subgraph
	isomorphic to a graph~$H$;
	in such a case we will abbreviate this by saying that $G$ \emph{conduces}~$H$.
	All graphs we consider are finite, simple, and undirected.
	
	The classical result of Ramsey~\cite{ramsey}, which served as a motivation for this paper
	and many others, is the following:
	
	\begin{theorem}[Ramsey's Theorem]\label{thm:ramsey}
		For every positive integer $r$, there is an integer $f_{\ref{thm:ramsey}}(r)$
		such that every graph on at least $f_{\ref{thm:ramsey}}(r)$ vertices conduces~$K_r$
		(a complete graph on $r$ vertices) or~$\overline K_r$ (an edgeless graph on $r$ vertices).
	\end{theorem} 
	There are numerous extensions of Ramsey's Theorem for graphs of various levels of connectivity and different relations on graphs.    
	For connected graphs, we have the following: 
	\begin{theorem}[(5.3) of \cite{Unavoidableminors3connbinmatroids}]\label{thm:diestel}
		For every positive integer $r$, there is an integer $f_{\ref{thm:diestel}}(r)$
		such that every connected graph on at least $f_{\ref{thm:diestel}}(r)$ vertices
		conduces one of the following graphs: $K_r$, $K_{1,r}$, and $P_r$.
	\end{theorem}
	
	For 2-connected graphs, we have the following for the relation of topological minors:
	
	\begin{theorem}[(1.2) of \cite{Unavoidabletopminor3conngraphs}]\label{thm:oot}
		For every integer $r$ exceeding two, there is an integer $f_{\ref{thm:oot}}(r)$ such that every $2$-connected graph on at least $f_{\ref{thm:oot}}(r)$ vertices contains a subgraph isomorphic to a subdivision of $K_{2,r}$ or~$C_r$. 
	\end{theorem}
	%\begin{theorem}[(1.5) of \cite{unavoidable par minor 4-conn graphs}] \label{thm:pm}
	%	There exists a function $f_{\ref{thm:pm}}(p,q,r,s)$ such that every $2$-connected graph on at least $f_{\ref{thm:pm}}(p,q,r,s)$ vertices conduces a parallel minor $K_p$, $K_{2,q}^+$, $C_r$, or $F_s$. 
	%\end{theorem}
	
	For topological minors, a theorem of this type was proved in \cite{Unavoidabletopminor3conngraphs} for 3- and internally-4-connected graphs. 
	For parallel minors, a theorem of this type was proved in \cite{unavoidableparminor4conngraphs} for 1-, 2-, 3-, and internally-4-connected graphs.
	Similar results have been proved for 3-connected binary and general matroids in \cite{Unavoidableminors3connbinmatroids} and \cite{Unavoidableminors3connmatroids}.
	The goal of the paper is to present an analogous result to \cref{thm:oot} using the original relation in Ramsey's Theorem of induced subgraphs. 
	
	Before stating precisely the main result of this paper, we need to define two families of graphs.
	Let $r$ be an integer exceeding two.
	Let $\mathscr{K}_{2,r}$ be the family of graphs obtained from $K_{2,r}$ by subdividing each of the edges of $K_{2,r}$ an arbitrary number, possibly zero, of times.  % I wrote "eaach of the edges" in an effort to clarify that each edge may be subdivided a different number of times.
	Let $\mathscr{K}_{2,r}^+$ be the family of graphs obtained from the family $\mathscr{K}_{2,r}$ by adding an edge between the two vertices of degree $r$ to each member of the family~$\mathscr{K}_{2,r}$.
	
	Trees and paths will play a significant role in this paper, so we need some definitions describing their properties.
	A tree $T$ with a distinguished vertex $\rho$, called the \emph{root}, is a \emph{rooted tree} and is denoted by~$(T,\rho)$.
	Its \emph{height} is the maximum distance from one of its vertices to the root.
	There is a natural partial ordering of the vertices of $T$:
	we write $u\le_T v$ whenever $u$ lies on the $\rho v$-path of~$T$.
	We write $u <_T v$ whenever $u$ lies on the $\rho v$-path of $T$ and $u$ is distinct from~$v$.
	If the identity of the tree is clear from the context, we may use $\le$ or~$<$ instead.
	The vertices $v$ such that $u<_{T}v$ are called the \emph{descendants} of~$u$.
	The descendants of $u$ that are also its neighbors are called its \emph{children}.
	For two vertices $a$ and $b$ of $T$ such that $a\le b$, the subgraph of $T$ induced by the vertices $v$ such that $a\le v\le b$ is denoted by~$T[a,b]$.
	Note that if in particular $a>b$, then $T[a,b]$ is empty.
	Similarly, the subgraph of $T$ induced by the vertices $v$ such that $a<v<b$ is denoted by~$T(a,b)$.
	The subgraphs $T(a,b]$ and $T[a,b)$ are defined analogously.
	
	A \emph{messy ladder} is a triple $(L,X,Y)$ that consists of a graph $L$ whose vertices all lie on two disjoint induced paths $X$ and $Y$, called \emph{rails}.
	Each rail is considered to be a tree rooted at one of its endpoints, which is called the \emph{initial vertex}, and the other endpoint is called the \emph{terminal vertex}.
	The edges of $L$ that belong to neither $X$ nor $Y$ are called \emph{rungs}.
	The graph $L$ has an edge, called $\sigma$, between the initial vertices of the rails, and an edge, called~$\tau$, between the terminal vertices of the rails.
	At most one of the rails may be trivial.
	In some contexts when we say messy ladder, we mean only the graph $L$, of which the existence and properties of $X$ and $Y$ are a part.
	The \emph{order} of a messy ladder is its number of vertices.
	The following are equal: the order of a messy ladder,  the order of the graph $L$, and the number of vertices in $X\cup Y$.
	
	If $e$ is a rung in a ladder with rails $X$ and $Y$, then $e_X$ and $e_Y$ denote
	the end\nobreakdash-vertices of $e$ on $X$ and $Y$, respectively.
	Two rungs in an ordered pair $e$ and $f$ \emph{cross} if $e_X<f_X$ and $f_Y <e_Y$.
	We also say that $(e,f)$ is a \emph{cross} whose \emph{$X$-span} is $X[e_X,f_X]$, 
	and whose \emph{$Y$-span} is~$Y[f_Y,e_Y]$. 
	A cross whose $X$-span and $Y$-span are both single edges is \emph{degenerate}.
	A \emph{clean ladder} is a messy ladder whose crosses are all degenerate. 
	
	In all figures of this paper, thick segments represent induced paths with an arbitrary number of vertices, while thin lines indicate single edges.
\begin{figure}[H]
	\begin{center}
	\begin{tikzpicture}
	[scale=.45,auto=left,every node/.style={circle, fill, inner sep=0 pt, minimum size=1mm, outer sep=0pt},line width=.4mm]
	\node (1) at (1,5) {};
	\node (2) at (1,1){};
	\node (3) at (18,5)  {};
	\node (4) at (18,1)  {};
	\node (5) at (7,5)[label=above: $e_X$] {};\node (6) at (8,5)[label=above: $f_X$]{};\node (7) at (7,1)[label=below: $f_Y$]{}; \node (8) at (8,1)[label={[label distance=2pt] below: $e_Y$}]{};
	\node (9) at (12.75,5){}; \node (10) at (13.5,5){}; \node (11) at (12.75,1){}; \node(12) at (13.5,1){};
	\node (13) at (14.25,1){};\node (14) at (14.25,5){};
	\draw[line width=.8mm] (1)-- (5);\draw (5) to (6);\draw[line width=.8mm] (6)to (10.75,5); \draw[Green, line width=.8mm] (10.75,5) to(11.25,5); \draw[line width=.8mm] (11.25,5)--(9);\draw (9) to (10) to (14); \draw[line width=.8mm] (14) to (3);
	\draw[line width=.8mm] (2)-- (2,1); \draw[blue, line width=.8mm](2,1)--(4,1); \draw[line width=.8mm](4,1)--(7);\draw(7)--(8); \draw[line width=.8mm] (8)--(10.25,1); \draw[Green, line width=.8mm](10.25,1)--(12,1);\draw[line width=.8mm](12,1)--(11);\draw(11) to(12)to (13);\draw[line width=.8mm](13)--(4);
	\draw (1) --(2); \draw (3)--(4);
	\draw[color=blue] (3.5,5) to (2,1);\draw[color=blue] (3.5,5) to (3,1);\draw[color=blue] (3.5,5) to (4,1);
	\draw (4,1) to (4.5,5); \draw (4,1) to (5.25,5); \draw (4,1) to (6,5);	
	\draw[color=red] (5) to (8); \draw[color=red] (6) to (7);
	\draw (8.75,5) to (8.75,1);
	\draw (9.25,1) to (9.75, 5); \draw(9.75,5) to (10.25,1); 
	\draw (12,1) to (11.75,5);\draw (12,1) to (12.25,5);\draw (12,1)to (12.75,5);
	\draw[color=Green] (12,1) to (11.25,5); \draw[Green] (10.25,1) to (10.75,5);
	\draw (9) to (12); \draw (11) to (10);
	\draw (12) to (14.25,5); \draw (10) to (14.25,1);
	
	\draw (17.25,1) to (3);
	\end{tikzpicture}
	\vspace{-1mm}
	\caption{A clean ladder}
	
	\label{fig:ladderd}
	\end{center}
\end{figure}
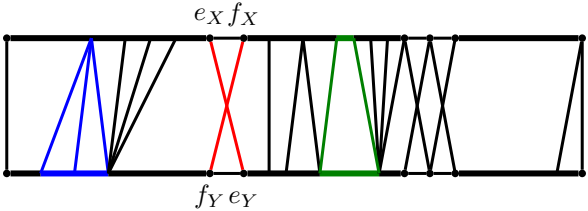
In \cref{fig:ladderd}, there are a few features to notice: the fan indicated by blue line segments is a clean ladder, and so is the cycle indicated by green line segments. 
A degenerate cross is depicted by the red line segments.
These structures are discussed in detail in \cref{Messy Ladder to Clean Ladder}.

The following is the main result of the paper:
\begin{theorem}\label{thm:finite}
	Let $r$ be an integer exceeding two. 
	There is an integer $f_{\ref{thm:finite}}(r)$ such that every $2$-connected graph of order at least $f_{\ref{thm:finite}}(r)$ conduces one of the following: $K_r$, a clean ladder of order at least $r$, a member of $\mathscr{K}_{2,r}$, and a member of $\mathscr{K}_{2,r}^+$.
\end{theorem}
Note that a clean ladder in the theorem can be replaced by a long cycle, a long fan (where rim edges could be subdivided) and a restricted version of the clean ladder where all rungs belong to degenerate crosses.

Our proof uses Ramsey numbers, the known bounds on which are believed to be very far from best possible.
So in our proofs, we value clarity of the arguments over the tightness of the bounds.

To prove the main theorem, we consider the cases where the large 2-connected graph $G$ either has a long path as a subgraph or it does not.  
Section \ref{Graphs Without a Long Path} discusses the case where $G$ does not have a long path. 
In that case, we prove that $G$ conduces two of the graphs listed in the conclusion of \cref{thm:finite}. 
The case where $G$ has a long path is broken into two sections.
In \cref{Long Path to Messy Ladder}, we start the with long path and obtain a large messy ladder.
In \cref{Messy Ladder to Clean Ladder}, we show that if a messy ladder is large enough, then it conduces a sufficiently large clean ladder.
\cref{Proving Main Theorem} combines the results of Sections~\cref{Graphs Without a Long Path,Long Path to Messy Ladder,Messy Ladder to Clean Ladder} to prove \cref{thm:finite}. 

\section{Graphs Without a Long path}\label{Graphs Without a Long Path}
In this section, we prove that a large 2-connected graph either has a long path or conduces one of the graphs in the main result.

A rooted tree $(T,\rho)$ that is a spanning subgraph of a graph $G$ is called \emph{normal}
if, for every two adjacent vertices $u$ and $v$ of $G$, either $u\le_{T}v$ or~$v\le_{T}u$.  
It is well known that every connected graph has a normal spanning tree (Proposition~1.5.6 of~\cite{diestel}).	
A \emph{rooted sub-tree} $(T',\rho')$ of $(T,\rho)$ has $T'$ as sub-tree of $T$ and $(T',\rho')$ preserves the ordering of~$(T,\rho)$.  

\begin{lemma}\label{lem:shortp}
	Let $q$ and $r$ be integers exceeding one.
	There is an integer $f_{\ref{lem:shortp}}(q,r)$ such that if $G$ is a $2$-connected graph on at least
	$f_{\ref{lem:shortp}}(q,r)$ vertices, then $G$ has either a path of order $q+1$ or an induced subgraph that is a member of one
	of the following families: $\mathscr{K}_{2,r}^+$ and $\mathscr{K}_{2,r}$.  
\end{lemma}
\begin{proof} 	
	We prove that $f_{\ref{lem:shortp}}(q,r)=2+(d-1)+(d-1)^2+\ldots +(d-1)^{q-1}$, where $d=1+(q-2)(r-1)$, satisfies the conclusion.	 
	
	Let $(T,\rho)$ be a normal spanning rooted tree of~$G$.  	
	If $(T,\rho)$ has height at least $q$, then $(T,\rho)$ has a path of order $q+1$, and the conclusion follows.
	
	For the remainder of the proof, we may assume the height of $(T,\rho)$ is less than~$q$.  		 
	Since $G$ has order at least $f_{\ref{lem:shortp}}(q,r)$, the tree $(T,\rho)$ has a vertex $v$ with at least $d$ children.
	
	Let $R$ be the $\rho v$ path in $(T,\rho)$, which has order at most~$q-1$. 
	For each child $v_i$ of $v$, let $(T_i,v_i)$ be the rooted sub-tree of $(T,\rho)$ induced by $v_i$ and all of the descendants of~$v_i$. 
	Since $v$ has at least $d$ children, there are at least $d$ sub-trees of $(T,\rho)$ rooted at children of $v$. 
	We need to consider only $d$ of them: $(T_1,v_1)$, $(T_2,v_2)$,~\dots,~$(T_d,v_d)$.  
	Since $(T,\rho)$ is normal and the rooted sub-trees are distinct, every edge of $G$ with exactly one end in some $T_i$ must have the other end in $R-v$.  
	Since $G$ is 2-connected, it follows that $v$ is not a cut-vertex of~$G$.  
	For each $j\in\{1,2,\dots,d\}$, there is an edge $e_j$ in $G$ incident with both a vertex on $(T_j,v_j)$ and a vertex $u_j$ on~$R-v$.
	By the definition of~$d$, there is a $k \in \{1,2,\dots,d\}$ such that $u_{k}$ is incident to at least $r$ of the edges $e_j$; let $u=u_{k}$.  
	Let~$\mathscr{I}$ be a set of $r$ indices from $\{1,2,\dots,d\}$ of the edges $e_i$ that have $u$ as one endpoint and the other endpoint on $(T_i,v_i)$.
	Each $(T_i,v_i)$ spans a component $G_i$ of $G-V(R)$.
	Both vertices $u$ and $v$ have neighbors in $G_i$.
	Let $G_i'$ be the subgraph of $G$ that consists of $G_i$ and all the edges between $G_i$ and $\{u,v\}$.
	Note that $G_i'$ is connected.
	Let $P_i$ be a shortest $uv$-path in $G_i'$.
	
	Let $H$ be the subgraph of $G$ induced by~$\bigcup_{i\in \mathscr{I}} P_i$.  
	Since $(T,\rho)$ is normal, $G$ has no edges between internal vertices of distinct paths in $\{P_i\}_{i\in \mathscr{I}}$,
	and since each $P_i$ is a shortest $uv$-path in $G_i$, it follows that $H$ is the union of internally-disjoint $uv$-paths.		   
	If $u$ is adjacent to $v$ in $G$, then $H$ is a member of the family~$\mathscr{K}_{2,r}^+$,
	and if $u$ is not adjacent to $v$ in $G$, then $H$ is a member of the family $\mathscr{K}_{2,r}$.
	The conclusion follows.
	 \begin{figure}[H]
	 	\centering
	 	\begin{subfigure}{.45\textwidth}
	 		\begin{center} 
	 		\begin{tikzpicture}
	 		[scale=.7,auto=left,every node/.style={circle, fill, inner sep=0 pt, minimum size=2.25mm, outer sep=0pt},line width=.4mm]
	 		\node (0) at (5.75,6.75)[label= {[label distance=2pt]right:$\rho$}]{};
	 		\node[fill=blue] (1) at (5.75,5.75)[label= {[label distance=2pt]above right:$u$}]{};
	 		\node (2) at (5.75,4.75){};
	 		\node[fill=blue] (3) at (5.75,3.75)[label= {[label distance=2pt]above right:$v$}]{};
	 		\node[fill=red] (4) at (3,2.5)[label= {[label distance=2pt]below left:$v_1$}]{};
	 		\node (5) at (4,2.5)[label= {[label distance=2pt]below left:$v_2$}]{};
	 		\node (6) at (5,2.5)[label= {[label distance=2pt]below:$v_3$}]{};
	 		\node[fill=red] (7) at (6.25,2.5)[label= {[label distance=2pt]below:$v_4$}]{};
	 		\node (8) at (7.25,2.5) [label= {[label distance=2pt]below left:$v_5$}]{};
	 		\node (9) at (8.25,2.5)[label= {[label distance=2pt]below right:$v_6$}]{};
	 		\node (10) at (3.75,1.5){};
	 		\node (11) at (7,1.5){};
	 		\node[fill=red] (13) at (7.5,1.5){};
	 		\node (12) at (4.25,1.5){};
	 		
	 		\draw (0)--(1)--(2)--(3);
	 		\draw (5)--(12); \draw[red] (8)--(13);
	 		\draw[red] (3)--(4); \draw (3)--(5); \draw (3)--(6); \draw[red] (3)--(7); \draw[red] (3)--(8);\draw (3)--(9);
	 		\draw (5)--(10); \draw (8)--(11);
	 		\draw[red] (1) to [bend right] node[midway, left=3pt, fill=white] {$e_1$} (4); \draw[red] (1) to [bend left] node[midway, right=3pt, fill=white] {$e_4$} (7); \draw[red] (1) to [bend left=45]node[midway, right=2pt, fill=white] {$e_5$}(13);
	 		\end{tikzpicture}
	 		\end{center}
	 		\caption{member of the family $\mathscr{K}_{2,r}$ in $G$}	
	 	\end{subfigure}
	 	\begin{subfigure}{.45\textwidth}
	 		\begin{center}
	 			\begin{tikzpicture}
	 		[scale=.7,auto=left,every node/.style={circle, fill, inner sep=0 pt, minimum size=2.25mm, outer sep=0pt},line width=.4mm]
	 		\node (0) at (5.75,6.75)[label= {[label distance=2pt]right:$\rho$}]{};
	 		\node (1) at (5.75,5.75){};
	 		\node[fill=blue] (2) at (5.75,4.75)[label= {[label distance=2pt]above right:$u$}]{};
	 		\node[fill=blue] (3) at (5.75,3.75)[label= {[label distance=2pt]above right:$v$}]{};
	 		\node (4) at (3,2.5)[label= {[label distance=2pt]below left:$v_1$}]{};
	 		\node (5) at (4,2.5)[label= {[label distance=2pt]below right:$v_2$}]{};
	 		\node[fill=red] (6) at (5,2.5)[label= {[label distance=2pt]below:$v_3$}]{};
	 		\node (7) at (6.25,2.5)[label= {[label distance=2pt]below:$v_4$}]{};
	 		\node (8) at (7.25,2.5) [label= {[label distance=2pt]below left:$v_5$}]{};
	 		\node[fill=red] (9) at (8.25,2.5)[label= {[label distance=2pt]below right:$v_6$}]{};
	 		\node[fill=red] (10) at (3.75,1.5){};
	 		\node (12) at (4.25,1.5){};
	 		\node (13) at (7.5,1.5){};
	 		\node (11) at (7,1.5){};
	 		\draw (0)--(1)--(2); \draw[blue](2)--(3);
	 		\draw (5)--(12); \draw(8)--(13);
	 		\draw (3)--(4); \draw[red] (3)--(5); \draw [red](3)--(6); \draw (3)--(7); \draw (3)--(8);\draw[red] (3)--(9);
	 		\draw[red] (5)--(10); \draw (8)--(11);
	 		\draw[red] (2) to [bend right=60] node[midway, left=3pt, fill=white] {$e_2$} (10); \draw[red] (2) to [bend right] node[midway, left=3pt, fill=white] {$e_3$} (6); \draw[red] (2) to [bend left]node[midway, right=3pt, fill=white] {$e_6$}(9);
	 		\end{tikzpicture}
	 		\end{center}
	 		\caption{member of the family $\mathscr{K}_{2,r}^+$ in $G$}	
	 		\label{fig:obtainingk2np}
	 	\end{subfigure}
 	\caption{Process of obtaining a member of the family $\mathscr{K}_{2,r}$ or $\mathscr{K}_{2,r}^+$}
 	\label{fig:obtainingk2n}  
	 \end{figure}
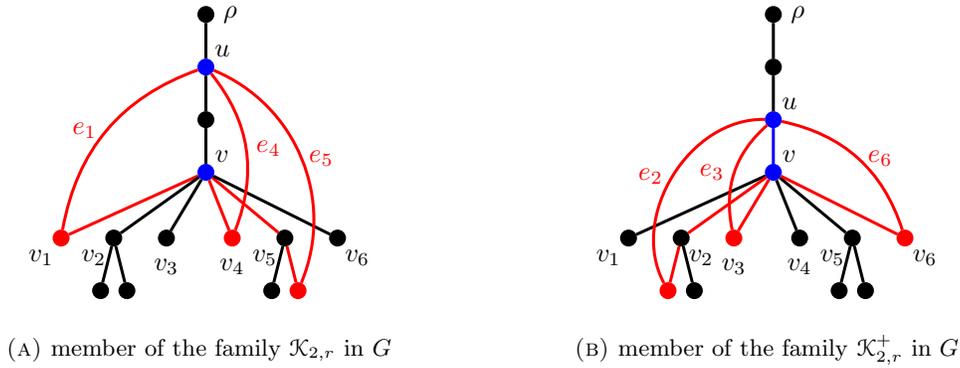 
   
 \cref{fig:obtainingk2n} shows the paths whose union is either a member of the family $\mathscr{K}_{2,r}$ or a member of the family~$\mathscr{K}_{2,r}^+$.  
 The red vertices are the vertices in the bipartition of size $r$ and the blue vertices are members of the bipartition of size two.  
 The red segments show the edges of a graph in $\mathscr{K}_{2,r}$ and the blue edge in \cref{fig:obtainingk2np} illustrates the edge between the two vertices of degree~$r$ in a member of the family $\mathscr{K}_{2,r}^+$.        
\end{proof}

\section{From a Long Path to a Messy Ladder}\label{Long Path to Messy Ladder}
In this section, we prove that if a large 2-connected graph $G$ has a long path as a subgraph, then $G$ conduces one of the following: a large messy ladder, a large complete graph, a large $K_{2,n}$, and a large $K_{2,n}^+$.
The goal of this section is to prove the following lemma.

\begin{lemma}\label{lem:longp}
	Let $p$ and $q$ be integers exceeding two. There is an integer  $f_{\ref{lem:longp}}(p,q)$ such that every $2$-connected graph with a path of order $f_{\ref{lem:longp}}(p,q)$ conduces one of the following: $K_p$, $K_{2,p}$, $K_{2,p}^+$, and a messy ladder of order at least~$q$.
\end{lemma}

Before proceeding, we need the following result of Galvin, Rival, and Sands~\cite{grs}.  
\begin{theorem}[Theorem 4 of \cite{grs}]\label{lem:traceable graphs}
	Let $p$, $q$, and $r$ be positive integers.  
	There is an integer $f_{\ref{lem:traceable graphs}}(p,q,r)$ such that every graph with a spanning path of order at least $f_{\ref{lem:traceable graphs}}(p,q,r)$ contains $K_{p,q}$ as a subgraph or conduces a path of order~$r$.
\end{theorem}
We use this theorem to prove that a large graph conduces either a graph from the list desired in \cref{thm:finite} or a long path.
\begin{corollary}\label{lem:longP}
	Let $q$ and $r$ be integers exceeding two.
	There is an integer $f_{\ref{lem:longP}}(q,r)$ such that every graph with a path of order at least $f_{\ref{lem:longP}}(q,r)$ conduces one of the following: $K_q$, $K_{2,q}$, $K_{2,q}^+$, and a path of order~$r$.
\end{corollary} 
\begin{proof}
	Let $f_{\ref{lem:longP}}(q,r)=f_{\ref{lem:traceable graphs}}(2,s,r)$ where $s=f_{\ref{thm:ramsey}}(q)$, and $f_{\ref{thm:ramsey}}(q)$ and $f_{\ref{lem:traceable graphs}}(2,s,r)$ are the numbers from Ramsey's Theorem (\cref{thm:ramsey}) and \cref{lem:traceable graphs}, respectively.
	We prove that $f_{\ref{lem:longP}}(q,r)$ satisfies the conclusion.
	Suppose $G$ is a graph with a path $P$ of order at least~$f_{\ref{lem:longP}}(q,r)$.

	Let $H$ be the graph obtained from $G$ by deleting all vertices except those on the path~$P$.  
	So $V(H)=V(P)$ and $H$ is an induced subgraph of~$G$.  
	Thus, the path $P$ is a spanning path of $H$ of order at least~$f_{\ref{lem:traceable graphs}}(2,s,r)$.
	By \cref{lem:traceable graphs}, the graph $H$ conduces a path of order $r$ or contains $K_{2,s}$ as a subgraph.  
	If $H$ conduces a path of order $r$, then so does $G$, and the conclusion follows.
	
	Therefore, we may assume $H$ has a subgraph isomorphic to $K_{2,s}$ whose bipartition is $(A,B)$ with $|A|=2$ and $|B|=s$.
	Let $H(B)$ be the subgraph of $H$ induced by $B$, we apply Ramsey's Theorem (\cref{thm:ramsey}) to $H(B)$.
	By \cref{thm:ramsey}, the graph $H(B)$ conduces either $K_q$ or~$\overline{K_q}$.
	If $H(B)$ conduces $K_q$, then so does $H$, and the conclusion follows.
	If $H(B)$ conduces $\overline{K_q}$, then let $I$ be an independent set of order $q$ in~$H(B)$.
	Since $K_{2,s}$ is a subgraph of $H$ which is not necessarily induced, there may be an edge between the two vertices of $A$ in~$H$.
	The subgraph of $H$ induced by the vertex set $A\cup I$ is isomorphic either to $K_{2,q}$ if the vertices of $A$ are non-adjacent, or to $K_{2,q}^+$ otherwise.	
	Since $G$ conduces $H$, we have that $G$ conduces one of the following: $K_q$, $K_{2,q}$, and $K_{2,q}^+$, as desired.	   
\end{proof}

We will use Tutte's notion of a bridge found in \cite{tuttematroids}, see also \cite{handbookcombinatorics}, to build the messy ladder.
Define a \emph{$H$-bridge} (or a \emph{bridge of $H$}) to be a connected subgraph $B$ of $G\setminus E(H)$ that satisfies either one of the following two conditions:
\begin{enumerate}
	\item $B$ is a single edge with both endpoints in~$V(H)$.
	In this case, $B$ is called a \emph{degenerate bridge}.
	\item $B-V(H)$ is a connected component of $G-V(H)$; and $B$ also includes every edge of $G$ with one end point in $V(B)-V(H)$ and the other end point in~$H$.
\end{enumerate}
Note that every edge of $G\setminus E(H)$ belongs to exactly one $H$-bridge. 
Vertices that belong to both $B$ and $H$  are called \emph{vertices of attachment} of~$B$.

Suppose $G$ is a large 2-connected graph that has a long induced $uv$-path $P$.
Our goal is to use $P$  to form a large induced messy ladder.
Since $P$ is an induced path, it has no degenerate bridges.
For each bridge $B_i$ of $P$ in $G$,  let $u_i$ and $v_i$ be the two vertices of attachment of $B_i$ such that $P[u_i,v_i]$ includes all vertices of attachment of~$B_i$.
We call $P[u_i,v_i]$ the \emph{span} of~$B_i$.
A \emph{$P$-bridge chain} $B_1$, $B_2$, \dots, $B_k$ is a sequence of bridges of an induced $uv$-path $P$ satisfying the following:
\[u=u_1< u_2< v_1\le u_3< v_2\le u_4< v_3 \le \dots \le u_{k-1}< v_{k-2}\le u_k< v_{k-1}< v_k\le v\]
The \emph{rank} of a $P$-bridge chain is the number of bridges that form the $P$-bridge chain.
The \emph{span} of a $P$-bridge chain is the union of the spans of its elements.
A $P$-bridge chain of rank 6 is shown in \cref{fig:bridgechain}.
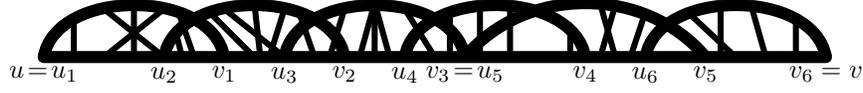
\begin{figure}[H]
	\centering
	\begin{tikzpicture}
		[scale=.8,auto=left,every node/.style={circle, fill, inner sep=0 pt, minimum size=.75mm, outer sep=0pt},line width=1.6mm]
		\node (1) at (1,2)[label= {[label distance=-8pt]below:$u\equal u_1$}] {};
		\node (2) at (3,2)[label=below:$u_2$] {};
		\node (4) at (4,2)[label=below:$v_1$] {};
		\node (5) at (5,2)[label=below:$u_3$] {};
		\node (6) at (6,2)[label=below:$v_2$] {};
		\node (7) at (7,2)[label=below:$u_4$] {};
		\node (8) at (8,2)[label={[label distance=-9.5pt]below:$v_3\equal u_5$}] {};
		%\node (9) at (8,2)[label=below:$u_5$] {};
		\node (10) at (10,2)[label=below:$v_4$] {};
		\node (11) at (11,2)[label=below:$u_6$] {};
		\node (12) at (12,2)[label=below:$v_5$] {};
		\node (14) at (14,2)[label={[label distance =-9pt]below:$v_6=v$}] {};
		\draw[line cap=round] (1,2) to (14,2);			
		\draw[bend left=80](1,2) to (4,2);
		\draw[line width=.8mm] (1.5,2.67) to (1.5,2);\draw[line width=.8mm] (2,2.8)to (3,2);
		\draw[line width=.8mm] (3.25,2.72) to (3.5,2);
		\draw[line width=.8mm] (2.5,2.9) to (2.5,2);
		\draw[line width=.8mm] (3,2.8) to (2,2);
		\draw[bend left=80] (3,2) to(6,2);
		\draw[line width=.8mm] (3.75,2.75) to (4.5,2);
		\draw[line width=.8mm] (4,2.8) to (5,2);
		\draw[line width=.8mm] (5,2.8) to (4.75,2);
		\draw[line width=.8mm] (4.5,2.9) to (4.6,2);
		\draw [line width=.8mm] (5.25,2.72) to (5.5,2);
		\draw[line width=.8mm] (3.2,2.4) to (3.3,2);
		
		\draw[bend left=80] (5,2) to (8,2);
		\draw[line width=.8mm] (6.5,2.9) to (6.5,2);
		\draw[line width=.8mm] (6.5,2.9) to (6.25,2);
		\draw[line width=.8mm] (6.5,2.9) to (6.75,2);
		\draw [line width=.8mm] (6,2.8) to (5.75,2);
		\draw[line width=.8mm] (7,2.8) to (7.5,2);
		
		\draw[bend left=80](7,2) to (10,2);
		\draw[line width=.8mm] (7.75,2.7) to (7.75,2);
		\draw [line width=.8mm] (8.25,2.82) to (8.25,2);
		\draw[line width=.8mm] (8.75,2.82) to (8.75,2);
		%\draw[line width=.8mm] (9.25,2.7) to (9.25,2);
		\draw[line width=.8mm] (9.75,2.41) to (9.75,2);
		
		\draw[bend left=50] (8,2) to (12,2);
		\draw[line width=.8mm] (10.25,2.88) to (10.5,2);
		\draw [line width=.8mm] (10.5,2.87) to (10.25,2);
		\draw [line width=.8mm] (11,2.75) to (10.75,2);
		
		\draw[bend left=80] (11,2) to (14,2);
		\draw[line width=.8mm] (12.25,2.88) to (12.25,2);
		\draw [line width=.8mm] (11.75,2.7) to (12.5,2);
		\draw[line width=.8mm] (12.75,2.88) to (13,2);
		\draw[line width=.8mm] (13.5,2.6) to (13.5,2);
	\end{tikzpicture}
	
	\caption{$P$-bridge chain of rank 6}
	\label{fig:bridgechain}
\end{figure} 

In the next lemma, we prove that if a large 2-connected graph $G$ has a long induced path, then $G$ has a bridge with a long span, or $G$ has a bridge chain of large rank.
\begin{lemma}\label{lem:bridges}
	Let $r$ be an integer exceeding three. There is an integer $f_{\ref{lem:bridges}}(r)$ such that every 2-connected graph with an induced path $P$ of order at least $f_{\ref{lem:bridges}}(r)$ has a $P$-bridge with span of order at least $r-1$ or a $P$-bridge chain of rank at least~$r-2$.
	
\end{lemma}
\begin{proof}
	Let~$f_{\ref{lem:bridges}}(r)=(r-2)+(r-4)(r-4)+1$.
	Suppose $G$ is a 2-connected graph that conduces a path $P$ of order at least~$f_{\ref{lem:bridges}}(r)$.
	Let $u$ and $v$ be the endpoints of $P$ such that~$u< v$.
	If the span of a bridge of $P$ has order at least $r-1$, then the conclusion follows. 
	We may therefore assume that each bridge has span of order at most~$r-2$.
	
	There is a $P$-bridge chain $B_1$, $B_2$, \dots,~$B_j$ such that the span of each $P$-bridge $B_i$ is $P[u_i, v_i]$ for $1\le i \le j$ and $u_1=u$.
	If $j\ge r-2$, then the conclusion of the lemma follows, and
	we may thus assume that every $P$-bridge chain with $u_1=u$ has rank at most $r-3$.		
	
	Select a $P$-bridge chain $\mathscr{B}=B_1$, $B_2$, \dots,~$B_k$ with $u_1=u$ and maximum span.
	In order to find an upper bound on the order of the span of $\mathscr{B}$, note that the span of $B_1$ has at most $r-2$ vertices, $k\le r-3$, and the span of each of the bridges $B_2$, $B_3$, \dots,~$B_k$ contributes at most $r-4$ new vertices. 
	Since $P$ has order at least $f_{\ref{lem:bridges}}(r)$, it follows that $v_k\ne v$. 
	Moreover, as $v_k$ is not a cut-vertex of $G$, the path $P$ has a bridge $B$ with a vertex of attachment on $P[u,v_k)$ and another vertex of attachment on $P(v_k,v]$.
	Let $\ell$ be minimal subject to $B$ having a vertex of attachment on $P[u,v_{\ell})$.
	The $P$-bridge chain $B_1$, $B_2$, \dots,~$B_{\ell}$, $B$ has larger span than $\mathscr{B}$; a contradiction.
	
	Thus $G$ has a $P$-bridge chain with rank at least $r-2$, as required.
\end{proof}

The next lemma proves that in either outcome of \cref{lem:bridges}, the graph under consideration conduces a large messy ladder. 

\begin{lemma}\label{lem:longPmessyL}
	Let $r$ be an integer exceeding three. 
	There is an integer $f_{\ref{lem:longPmessyL}}(r)$ such that if a 2-connected graph $G$ has an induced path of order $f_{\ref{lem:longPmessyL}}(r)$, then $G$ conduces a messy ladder of order at least~ $r$.
\end{lemma}
\begin{proof}
	Let $f_{\ref{lem:longPmessyL}}(r)=(r-2)+(r-4)(r-4)+1$, which is equal to the number $f_{\ref{lem:bridges}}(r)$ from \cref{lem:bridges}.  
	Suppose that $G$ has an induced path $P$ of order~$f_{\ref{lem:longPmessyL}}(r)$. 
	For each bridge $B_i$ of $P$ in $G$,  let $u_i,v_i$ be the two vertices of attachment of $B_i$ such that $P[u_i,v_i]$ is the span of~$B_i$.  	 
	
	If $G$ has a $P$-bridge $B$ with span $P[u',v']$ having order  at least~$r-1$, then let $Q$ be an induced path in $B$ with end\nobreakdash-vertices $u'$ and $v'$.
	Since $P$ is induced, the path $Q$ has at least one vertex distinct from $u'$ and $v'$.
	The subgraph of $G$ induced by $Q\cup P[u',v']$ is a messy ladder of order at least $r$ with rails $P(u',v')$ and $Q$.
	The conclusion follows.	 	
	
	Now, we may assume by \cref{lem:bridges} that $G$ has a $P$-bridge chain ~$B_1$,~$B_2$,~\dots,~$B_{r-2}$.
	For each $1\le i \le r-2$, let $Q_i$ be an induced path in $B_i$ with the endpoints $u_i$ and~$v_i$.  
	Since $P$ is induced, each $Q_i$ contains at least one vertex distinct from $u_i$ and~$v_i$.
	Define $G'$ to be the subgraph of $G$ induced by~$P[u_1,v_{r-2}]~\cup~ \bigcup\limits_{i=1}^{r-2} Q_i$.
	
	In $G'$, we delete vertices on $P(u_{j+1},v_j)$, if they exist, for $j=\{1,2,\dots, r-3\}$ to obtain a graph~$G''$.
	If $r-2$ is odd, then let 
	$X=Q_1\cup P(v_1,u_3)\cup Q_3\cup P(v_3,u_5)\cup Q_5\cup~\cdots~ \cup P(v_{r-4},u_{r-2})\cup Q_{r-2}$ and $Y=P(u_1,u_2)\cup Q_2\cup P(v_2,u_4)\cup Q_4\cup P(v_4,u_6)\cup~\cdots~\cup Q_{r-3}\cup P(v_{r-3},v_{r-2})$.
	If $r-2$ is even, then let $X=Q_1\cup P(v_1,u_3)\cup Q_3\cup P(v_3,u_5)\cup~\cdots~\cup Q_{r-3} \cup P(v_{r-3},v_{r-2})$ and $Y=P(u_1,u_2)\cup Q_2\cup P(v_2,u_4)\cup Q_4\cup~\cdots~\cup P(v_{r-4},u_{r-2})\cup Q_{r-2}$. 
	Let the root of $X$ be $u_1$ and let the root of $Y$ be the neighbor of $u_1$ on~$P$.
	Notice that all vertices of $G''$ lie on $X\cup Y$, the graph $G$ conduces $G''$, and that $(G'',X,Y)$ is a messy ladder.  
 	 	
 	\cref{fig:messyl} illustrates this process of obtaining a messy ladder $(G'',X,Y)$ from~$G'$. 
 	The rails $X$ and $Y$ of $(G'',X,Y)$ are indicated by the green and blue paths.
 	We remind the reader that thin line segments indicate edges of $G'$ and $G''$ and thick curves and line segments indicate induced paths of~$G'$ and $G''$, with the straight line segments possibly being trivial.
 	\medskip
 	\begin{figure}[H]
 		\centering
 		\begin{subfigure}[b]{0.47\textwidth}
 			\begin{tikzpicture}
 			[scale=.5,auto=left,every node/.style={circle, fill, inner sep=0 pt, minimum size=.75mm, outer sep=0pt},line width=.7mm]
 			\node (1) at (1,2) {};
 			\node (14) at (14,2) {};
 			\draw (1,2) to node[midway, below=2pt, fill=white] {$P$} (14,2);			
 			\draw[bend left=80](1,2) to node[midway, above=2pt, fill=white] {$Q_1$} (4,2);
 			\draw[line width=.35mm] (1.5,2.67) to (1.5,2);\draw[line width=.35mm] (2,2.8)to (3,2);
 			\draw[line width=.35mm] (3.25,2.72) to (3.5,2);
 			\draw[line width=.35mm] (2.5,2.9) to (2.5,2);
 			\draw[line width=.35mm] (3,2.8) to (2,2);
 			\draw[bend left=80] (3,2) to node[midway, above=2pt, fill=white] {$Q_2$}(6,2);
 			\draw[line width=.35mm] (3.75,2.75) to (4.5,2);
 			\draw[line width=.35mm] (4,2.8) to (5,2);
 			\draw[line width=.35mm] (5,2.8) to (4.75,2);
 			\draw[line width=.35mm] (4.5,2.9) to (4.6,2);
 			\draw [line width=.35mm] (5.25,2.72) to (5.5,2);
 			\draw[line width=.35mm] (3.2,2.4) to (3.3,2);
 			
 			\draw[bend left=80] (5,2) to node[midway, above=2pt, fill=white] {$Q_3$}(8,2);
 			\draw[line width=.35mm] (6.5,2.9) to (6.5,2);
 			\draw[line width=.35mm] (6.5,2.9) to (6.25,2);
 			\draw[line width=.35mm] (6.5,2.9) to (6.75,2);
 			\draw [line width=.35mm] (6,2.8) to (5.75,2);
 			\draw[line width=.35mm] (7,2.8) to (7.5,2);
 			
 			\draw[bend left=80](7,2) to (10,2);
 			\draw[line width=.35mm] (7.75,2.7) to (7.75,2);
 			\draw [line width=.35mm] (8.25,2.82) to (8.25,2);
 			\draw[line width=.35mm] (8.75,2.82) to (8.75,2);
 			\draw[line width=.35mm] (9.25,2.7) to (9.25,2);
 			\draw[line width=.35mm] (9.75,2.41) to (9.75,2);
 			
 			\draw[bend left=80] (9,2) to (12,2);
 			\draw[line width=.35mm] (10.25,2.83) to (10.5,2);
 			\draw [line width=.35mm] (10.5,2.88) to (10.25,2);
 			\draw [line width=.35mm] (11,2.85) to (10.75,2);

 			\draw[bend left=80] (11,2) to (14,2);
 			\draw[line width=.35mm] (12.25,2.88) to (12.25,2);
 			\draw [line width=.35mm] (11.75,2.7) to (12.5,2);
 			\draw[line width=.35mm] (12.75,2.88) to (13,2);
 			\draw[line width=.35mm] (13.5,2.6) to (13.5,2);
 			\end{tikzpicture}
 			\caption{$G'$ where $r-2$ is even}
 		\end{subfigure}
 		\begin{subfigure}[b]{.47\textwidth}
 			\begin{tikzpicture}
 			[scale=.5,auto=left,every node/.style={circle, fill, inner sep=0 pt, minimum size=.75mm, outer sep=0pt},line width=.7mm]
 			\node (1) at (1,2) {};
 			\node (14) at (14,2) {};
 			\draw (1,2) to (14,2);			
 			\draw[bend left=80](1,2) to (4,2);
 			\draw[line width=.35mm] (1.5,2.67) to (1.5,2);\draw[line width=.35mm] (2,2.8)to (3,2);
 			\draw[line width=.35mm] (3.25,2.72) to (3.5,2);
 			\draw[line width=.35mm] (2.5,2.9) to (2.5,2);
 			\draw[line width=.35mm] (3,2.8) to (2,2);
 			\draw[bend right=80] (3,2) to (6,2);
 			\draw[line width=.35mm] (3.75,1.25) to (4.5,2);
 			\draw[line width=.35mm] (4,1.2) to (5,2);
 			\draw[line width=.35mm] (5,1.2) to (4.75,2);
 			\draw[line width=.35mm] (4.5,1.1) to (4.6,2);
 			\draw [line width=.35mm] (5.25,1.28) to (5.5,2);
 			\draw[line width=.35mm] (3.2,1.6) to (3.3,2);
 			
 			\draw[bend left=80] (5,2) to(8,2);
 			\draw[line width=.35mm] (6.5,2.9) to (6.5,2);
 			\draw[line width=.35mm] (6.5,2.9) to (6.25,2);
 			\draw[line width=.35mm] (6.5,2.9) to (6.75,2);
 			\draw [line width=.35mm] (6,2.8) to (5.75,2);
 			\draw[line width=.35mm] (7,2.8) to (7.5,2);
 			
 			\draw[bend right=80](7,2) to (10,2);
 			\draw[line width=.35mm] (7.75,1.3) to (7.75,2);
 			\draw [line width=.35mm] (8.25,1.18) to (8.25,2);
 			\draw[line width=.35mm] (8.75,1.18) to (8.75,2);
 			\draw[line width=.35mm] (9.25,1.3) to (9.25,2);
 			\draw[line width=.35mm] (9.75,1.59) to (9.75,2);
 			
 			\draw[bend left=80] (9,2) to (12,2);
 			\draw[line width=.35mm] (10.25,2.83) to (10.5,2);
 			\draw [line width=.35mm] (10.5,2.88) to (10.25,2);
 			\draw [line width=.35mm] (11,2.85) to (10.75,2);

 			\draw[bend right=80] (11,2) to (14,2);
 			\draw[line width=.35mm] (12.25,1.12) to (12.25,2);
 			\draw [line width=.35mm] (11.75,1.3) to (12.5,2);
 			\draw[line width=.35mm] (12.75,1.12) to (13,2);
 			\draw[line width=.35mm] (13.5,1.4) to (13.5,2);
 			
 			\end{tikzpicture}
 			\caption{A nicer representation of $G'$}
 		\end{subfigure}\\
 		\begin{subfigure}[b]{.47\textwidth}
 			\begin{tikzpicture}
 			[scale=.5,auto=left,every node/.style={circle, fill, inner sep=0 pt, minimum size=.75mm, outer sep=0pt},line width=.7mm]
 			\node (1) at (1,2) {};
 			\node (14) at (14,2) {};
 			\draw[line width=.35](1,2) to (1.4,2); 
 			\draw[line width=.35] (13.7,2) to (14,2);	
 			\draw[line width=.35mm] (1.5,2.67) to (1.5,2);
 			\draw[line width=.35mm] (2,2.8)to (3,2);
 			\draw[line width=.35mm] (2.5,2.9) to (2.5,2);
 			\draw[line width=.35mm] (3,2.8) to (2,2);
 			\draw[bend left=80,color=Green](1,2) to (4,2);
 			
 			\draw[line width=.35mm] (3.75,1.25) to (4.5,2);
 			\draw[line width=.35mm] (4,1.2) to (5,2);
 			\draw[line width=.35mm] (5,1.2) to (4.75,2);
 			\draw[line width=.35mm] (4.5,1.1) to (4.6,2);
 			\draw[bend right=80,color=blue] (3,2) to (6,2);
 			
 			\draw[line width=.35mm] (6.5,2.9) to (6.5,2);
 			\draw[line width=.35mm] (6.5,2.9) to (6.25,2);
 			\draw[line width=.35mm] (6.5,2.9) to (6.75,2);
 			\draw[bend left=80,color=Green] (5,2) to(8,2);
 			
 			\draw [line width=.35mm] (8.25,1.18) to (8.25,2);
 			\draw[line width=.35mm] (8.75,1.18) to (8.75,2);
 			\draw[bend right=80,color=blue](7,2) to (10,2);
 			
 			\draw[line width=.35mm] (10.25,2.83) to (10.5,2);
 			\draw [line width=.35mm] (10.5,2.88) to (10.25,2);
 			\draw [line width=.35mm] (11,2.85) to (10.75,2);
 			\draw[bend left=80,color=Green] (9,2) to (12,2);
 			
 			\draw[line width=.35mm] (12.25,1.12) to (12.25,2);
 			\draw [line width=.35mm] (11.75,1.3) to (12.5,2);
 			\draw[line width=.35mm] (12.75,1.12) to (13,2);
 			\draw[line width=.35mm] (13.5,1.4) to (13.5,2);
 			\draw[line width=.35mm, ] (11,2) to (12,2);
 			\draw[bend right=80,color=blue] (11,2) to (14,2);
 			\draw[line width=.35mm](1,2) to(1.4,2);
 			\draw[line width=.35mm] (13.7,2) to (14,2);
 			\draw[color=blue] (1.4,2) to (3,2); \node[fill=blue] (2) at (3,2){};
 			\draw[color=Green] (4,2) to (5,2); \node[fill=Green] (3) at (4,2){}; \node[fill=Green] (5) at (5,2){};
 			\draw[color=blue] (6,2) to (7,2); \node [fill=blue](6) at (6,2){}; \node[fill=blue] (5) at (7,2){};
 			\draw[color=Green] (8,2) to (9,2);\node [fill=Green](3) at (8,2){}; \node[fill=Green] (5) at (9,2){};
 			\draw[color=blue] (10,2) to (11,2);\node[fill=blue] (3) at (10,2){}; \node[fill=blue] (5) at (11,2){};
 			\draw[color=Green] (12,2) to (13.7,2);\node[fill=Green] (3) at (12,2){}; 
 			\end{tikzpicture}
 			\caption{$(G'',X,Y)$}
 		\end{subfigure}
 		\caption{Process of obtaining a messy ladder from $G'$ }
 		\label{fig:messyl}
 	\end{figure}
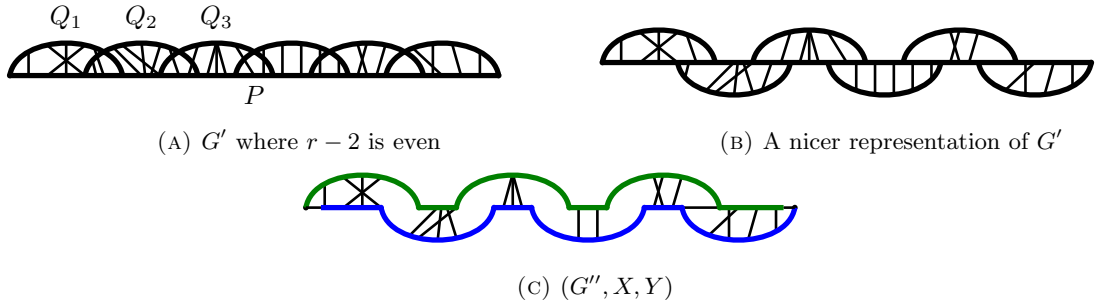
 
 	 	Since each of the $r-2$ bridges contributes to the messy ladder at least one vertex not on $P$, it follows that $(G'',X,Y)$ is a messy ladder of order at least~$r$, as required. 	
 \end{proof}
Note that the numbers in the conclusion of two previous lemmas are the same.  
The process of obtaining a messy ladder from a long induced path has been described in two steps, namely \cref{lem:bridges,lem:longPmessyL}.

We are now ready to prove \cref{lem:longp}, restated below.
\begin{namedthm*}{\cref{lem:longp}}
Let $p$ and $q$ be integers exceeding two. There is an integer  $f_{\ref{lem:longp}}(p,q)$ such that every $2$-connected graph with a path of order $f_{\ref{lem:longp}}(p,q)$ conduces one of the following: $K_p$, $K_{2,p}$, $K_{2,p}^+$, and a messy ladder of order at least~$q$.
\end{namedthm*} 	
\begin{proof}
Let $f_{\ref{lem:longp}}(p,q)= f_{\ref{lem:longP}}(p,r)$ where ~$r=~f_{\ref{lem:longPmessyL}}(q)$.  
Since $G$ has a path of order at least $f_{\ref{lem:longP}}(p,r)$, the graph $G$ conduces one of the following $K_p$, $K_{2,p}$, $K_{2,p}^+$, and~$P_{r}$. 
In the last case, \cref{lem:longPmessyL} implies that $G$ conduces a messy ladder of order at least~$q$, as required.
\end{proof}

\section{From a Messy Ladder to a Clean Ladder}\label{Messy Ladder to Clean Ladder}
We prove that a sufficiently large messy ladder conduces a clean ladder of the desired order.

In order to clean the ladder, we need to define some terms for the crosses.  
The cross $(e,f)$ is \emph{full} if the messy ladder $(L,X,Y)$ has no other cross whose $X$-span contains
the $X$-span of $(e,f)$ and whose $Y$-span contains the $Y$-span of~$(e,f)$.
Two crosses are \emph{independent} if both their $X$-spans and $Y$-spans are edge-disjoint.

In general, crosses may not be ordered in any particular way with respect to the rails $X$ and $Y$, however, \indfull crosses may be ordered by the position in which their vertices appear on the rails, as explained in \cref{l:max-cross-order}.
\begin{lemma}\label{l:max-cross-order}
Let $(e,f)$ and $(g,h)$ be  \indfull crosses of a messy ladder $(L,X,Y)$,
with the $X$- and $Y$-spans being $X[e_X,f_X]$, $Y[f_Y,e_Y]$,
$X[g_X,h_X]$, and $Y[h_Y,g_Y]$, respectively.
Then $f_X \le g_X$ if and only if~$e_Y \le h_Y$.
\end{lemma}

\begin{proof}
Let $(e,f)$ and $(g,h)$ be \indfull crosses.
Suppose for a contradiction that $f_X\le g_X$, however~$e_Y\not\le h_Y$.
Then $h_Y\le f_Y<e_Y$ and ~$e_X<f_X\le g_X< h_X$.
Thus $(e,h)$ is a cross whose $X$-span contains the $X$-spans of $(e,f)$ and of~$(g,h)$.  
This contradicts the fact that the crosses $(e,f)$ and~$(g,h)$ are full.
Hence~$e_Y\le h_Y$.
The other direction of the proof follows an analogous argument.
\end{proof}

For two pairwise independent full crosses $(e,f)$ and $(g,h)$, define the relation $(e,f)<(g,h)$ by $f_X\le g_X$ (or~$e_Y\le~h_Y$). 

We will be interested in maximal sequences of \indfull crosses, that is, those sequences that do not appear as proper subsequences of any other sequence of \indfull crosses.

Let $(L,X,Y)$ be a messy ladder with $\sigma_X$, $\tau_X$, $\sigma_Y$, and $\tau_Y$ as the initial and terminal vertices of $X$ and $Y$, respectively, and let $\mathscr{X}$ be a maximal sequence of \indfull crosses in~$(L,X,Y)$.
Our goal is now to use $\mathscr{X}$ to eliminate all non-degenerate crosses in the messy ladder $(L,X,Y)$ to obtain a clean ladder~$(H,U,W)$.
To do this, we need the following operation on ladders that eliminates non-degenerate \indfull crosses.

Let $\mathscr{X}=(\X_1, \X_2,\dots,\X_z)$ be a maximal sequence of \indfull crosses in $(L,X,Y)$.
The operation of \emph{resolving the cross $\X_i=(e^i,f^i)$} results in a triple $(L',X',Y')$ where~$L'=~L-X(e_X^i,f_X^i)- ~Y(f_Y^i,e_Y^i)$ and $X'=X[\sigma_X,e^i_X]\cup \{e^i\} \cup Y[e^i_Y,\tau_Y]$ and~$Y'=Y[\sigma_Y,f^i_Y]\cup \{f^i\}\cup X[f^i_X,\tau_X]$. 
Since $\X_i$ is a full cross, the graph $L$ has rungs neither from $X[\sigma_X,e^i_X]$ to $Y[e^i_Y,\tau_Y]$ nor from $Y[\sigma_Y,f^i_Y]$ to~$X[f^i_X,\tau_X]$. 
Thus $X'$ and $Y'$ are induced in $L'$, and $(L',X',Y')$ is a messy ladder.
If $\X_i$ is degenerate, then resolving the cross $\X_i$ results in the edges $e_X^if_X^i$ and $f_Y^ie_Y^i$ becoming rungs of $L'$, and the rungs $f^i$ and $e^i$ becoming edges on the rails $X'$ and $Y'$.
Note that we have not deleted any edges or vertices in this case, so the messy ladders $(L',X',Y')$ and $(L,X,Y)$ are isomorphic.

For a maximal sequence of pairwise \indfull crosses $\mathscr{X}=(\X_1,\X_2,\dots, \X_z)$ where $\X_i=(e^i,f^i)$ of a messy ladder $(L,X,Y)$, we inductively define the triples that result from resolving consecutive crosses of~$\mathscr{X}$.
Let $(L^1,X^1,Y^1)$ be the messy ladder obtained by resolving the cross $\X_1$ with
$X^1 =X[\sigma_X,e^1_X]\cup\{e^1\}\cup Y[e^1_Y,\tau_Y]$ and~$Y^1=~Y[\sigma_Y,f^1_Y]\cup~\{f^1\}\cup X[f^1_X,\tau_X]$.
Since the crosses in $\mathscr{X}$ are pairwise independent, the operation of resolving $\X_1$ leaves the other crosses in $\mathscr{X}$ unchanged.
Since the cross $\X_1$ is full, the operation of resolving $\X_1$ does not create a  non-degenerate cross.
If $\X_1$ is degenerate, then $(L^1,X^1,Y^1)$ is isomorphic to~$(L,X,Y)$.
If $\X_1$ is not degenerate, then $(\X_2,\X_3,\dots,\X_z)$ is a maximal sequence of \indfull crosses in $(L^1,X^1,Y^1)$.

For the inductive process, the definition of the rails $X^i$ and $Y^i$ depends on the parity of~$i$.
Suppose we have defined $(L^{i-1},X^{i-1},Y^{i-1})$ for some $2\le i \le z$ where $(L^{i-1},X^{i-1},Y^{i-1})$ is the triple obtained by resolving the crosses~$(\X_1,\X_2,\dots,\X_{i-1})$.
Since each cross of $\mathscr{X}$ is full, the operation of resolving the crosses $(\X_1,\X_2,\dots,\X_{i-1})$ does not create non-degenerate crosses.
Since the crosses of $\mathscr{X}$ are pairwise independent, the crosses $(\X_i,\X_{i+1},\dots,\X_z)$ are unchanged by the operation of resolving the crosses $(\X_1,\X_2,\dots,\X_{i-1})$.
Each cross in $(\X_1,\X_2,\dots,\X_{i-1})$ that was degenerate in $(L,X,Y)$ remains degenerate after resolving the crosses ~$(\X_1,\X_2,\dots,\X_{i-1})$.
So the degenerate crosses from $\X_1$, $\X_2$,~\dots,~$\X_{i-1}$ together with the crosses $\X_i$, $\X_{i+1}$,~\dots,~$\X_z$ form a maximal sequence of \indfull crosses in ~$(L^{i-1},X^{i-1},Y^{i-1})$. 

Let $(L^i,X^i,Y^i)$ be the messy ladder obtained from $(L^{i-1},X^{i-1},Y^{i-1})$ by resolving~$\X_i$.
If $i$ is even, then let ${X^i=X[\sigma_X,e^1_X]\cup\{e^1\}\cup Y[e^1_Y,f^2_Y]\cup \{f^2\}\cup~\cdots~ \cup \{f^i\}\cup X[f^i_X,\tau_X]}$ and $Y^i=Y[\sigma_Y,f^1_Y]\cup \{f^1\}\cup X[f^1_X,e^2_X]\cup\{e^2\}\cup ~\cdots~ \cup \{e^i\}\cup Y[e^i_Y,\tau_Y]$.
If $i$ is odd, then let $X^i=X[\sigma_X,e^1_X]\cup \{e^1\}\cup Y[e^1_Y,f^2_Y]\cup \{f^2\}\cup~\cdots ~\cup\{e^i\}\cup Y[e^i_Y,\tau_Y]$ and $Y^i=Y[\sigma_Y,f^1_Y]\cup \{f^1\}\cup X[f^1_X,e^2_X]\cup\{e^2\}\cup~ \cdots ~\cup \{f^i\}\cup X[f^i_X,\tau_X]$.

Let~$(H,U,W)=(L^z,X^z,Y^z)$.
Since we have resolved the crosses of $\mathscr{X}$, every cross from $\mathscr{X}$ that is in $(H,U,W)$ is degenerate.
Thus $(H,U,W)$ is a clean ladder.
\begin{remark}\label{vertsofcl}
The vertices $e_X^1$, $e_Y^1$, $f_X^1$, $f_Y^1$, \dots ,  $e_X^z$, $e_Y^z$, $f_X^z$, and $f_Y^z$ of the independent full crosses of $\mathscr{X}$ are members of the vertex set of $(H,U,W)$.
\end{remark}
This process of resolving the crosses is depicted in \cref{fig:maxrungseq,fig:graphHUW} below.

	\begin{figure}[H] 
	\centering
	\begin{subfigure}[b]{0.5\textwidth}
		
		\begin{tikzpicture}
		[scale=.45,auto=left,every node/.style={circle, fill, inner sep=0 pt, minimum size=1mm, outer sep=0pt},line width=.4mm]
		\node (1) at (1,5) [label={[label distance=1pt]above: $\sigma_X$}] {};
		\node (2) at (1,1) [label={[label distance=1pt]below: $\sigma_Y$}]{};
		\node (3) at (2.25,5) [label={[label distance=1pt]above: $e_X^1$}]{};
		\node (4) at (2.25,1) [label={[label distance=1pt]below: $f_Y^1$}]{};
		\node (5) at (3.5,5) [label={[label distance=1pt]above: $f_X^1$}]{};
		\node (6) at (3.5,1) [label={[label distance=1pt]below: $e_Y^1$}]{};
		\node (7) at (5,5) [label={[label distance=1pt]above: $e_X^2$}]{};
		\node (8) at (5,1) [label={[label distance=1pt]below: $f_Y^2$}]{};
		\node (9) at (7,5) [label={[label distance=-7pt]above: $f_X^2,e_X^3$}]{};
		\node (10) at (7,1) [label={[label distance=-7pt]below: $e_Y^2,f_Y^3$}]{};
		\node (11) at (9.5,5) [label={[label distance=1pt]above: $f_X^3$}]{};
		\node (12) at (9.5,1) [label={[label distance=1pt]below: $e_Y^3$}]{};
		\node (13) at (11,5) [label={[label distance=1pt]above: $e_X^4$}]{};
		\node (14) at (12,1) [label={[label distance=1pt]below: $e_Y^4$}]{};
		\node (15) at (12,5) [label={[label distance=1pt]above: $f_X^4$}]{};
		\node (16) at (11,1) [label={[label distance=1pt]below: $f_Y^4$}]{};
		\node (17) at (14,5) [label={[label distance=1pt]above: $e_X^5$}]{};
		\node (18) at (14,1) [label={[label distance=1pt]below: $f_Y^5$}]{};
		\node (19) at (15.5,5) [label={[label distance=1pt]above: $f_X^5$}]{};
		\node (20) at (15.5,1) [label={[label distance=1pt]below: $e_Y^5$}]{};
		\node (21) at (17,5) [label={[label distance=1pt]above: $\tau_X$}] {};
		\node (22) at (18,1) [label={[label distance=1 pt]below: $\tau_Y$}] {};
		\draw (1)--(2);\draw (21)--(22);	
		\draw (3) to (6); \draw (5) to (4);
		\draw (7) to (10); \draw (9) to (8);
		\draw (9) to (12); \draw (11) to (10);
		\draw (15) to (16); \draw (13) to (14);
		\draw (18) to (19);\draw (20) to (17);
		\draw (2.75,1) to (2.75,5);
		\draw[red, loosely dashed] (1.5,1) to (4.25,5); \draw[red,dashed] (1.5,5)to (4,1);
		\draw (13,5) to (12.5,1); \draw (13,5) to (13,1); \draw (13,5) to (13.5,1);
		\draw (10.25,5) to (10.25,1);
		\draw[line width=.8mm] (1)--(13); \draw (13)--(15);\draw[line width=.8mm] (15)--(21);
		\draw[line width=.8mm] (2)--(16); \draw (14)--(16);\draw[line width=.8mm](14)--(22);
		\end{tikzpicture}		
		\caption{Sequence of independent full crosses}
	\label{fig:maxrungseq}
	\end{subfigure}	
	\begin{subfigure}[b]{0.49\textwidth}
		
		\begin{tikzpicture}
		[scale=.45,auto=left,every node/.style={circle, fill, inner sep=0 pt, minimum size=1mm, outer sep=0pt},line width=.4mm]
		\node (1) at (1,5) [label={[label distance=1pt]above: $\sigma_X$},fill=blue] {};
		\node (2) at (1,1) [fill=Green, label={[label distance=1pt]below: $\sigma_Y$}]{};
		\node (3) at (2.25,5) [fill=blue, label={[label distance=1pt]above: $e_X^1$}]{};
		\node (4) at (2.25,1) [fill=Green, label={[label distance=1pt]below: $f_Y^1$}]{};
		\node (5) at (3.5,5) [fill=Green, label={[label distance=1pt]above: $f_X^1$}]{};
		\node (6) at (3.5,1) [fill=blue, label={[label distance=1pt]below: $e_Y^1$}]{};
		\node (7) at (5,5) [fill=Green, label={[label distance=1pt]above: $e_X^2$}]{};
		\node (8) at (5,1) [fill=blue, label={[label distance=1pt]below: $f_Y^2$}]{};
		\node (9) at (7,5) [fill=blue, label={[label distance=-7pt]above: $f_X^2,e_X^3$}]{};
		\node (10) at (7,1) [fill=Green, label={[label distance=-7pt]below: $e_Y^2,f_Y^3$}]{};
		\node (11) at (9.5,5) [fill=Green, label={[label distance=1pt]above: $f_X^3$}]{};
		\node (12) at (9.5,1) [fill=blue,label={[label distance=1pt]below: $e_Y^3$}]{};
		\node (13) at (11,5) [fill=Green, label={[label distance=1pt]above: $e_X^4$}]{};
		\node (14) at (12,1) [fill=Green,label={[label distance=1pt]below: $e_Y^4$}]{};
		\node (15) at (12,5) [fill=blue, label={[label distance=1pt]above: $f_X^4$}]{};
		\node (16) at (11,1) [fill=blue, label={[label distance=1pt]below: $f_Y^4$}]{};
		\node (17) at (14,5) [fill=blue, label={[label distance=1pt]above: $e_X^5$}]{};
		\node (18) at (14,1) [fill=Green,label={[label distance=1pt]below: $f_Y^5$}]{};
		\node (19) at (15.5,5) [fill=Green, label={[label distance=1pt]above: $f_X^5$}]{};
		\node (20) at (15.5,1) [fill=blue, label={[label distance=1pt]below: $e_Y^5$}]{};
		\node (21) at (17,5) [fill=Green, label={[label distance=1pt]above: $\tau_X$}] {};
		\node (22) at (18,1) [fill=blue, label={[label distance=1pt]below: $\tau_Y$}] {};
		\draw (1)--(2);\draw (21)--(22);	
		%\draw[LSUP!50!white] (2.75,1) to (2.75,5);
		\draw[blue] (3) to (6); \draw[Green] (5) to (4);
		\draw [Green](7) to (10); \draw [blue](9) to (8);
		\draw[blue] (9) to (12); \draw [Green](11) to (10);
		\draw[blue] (15) to (16); \draw[Green] (13) to (14);
		\draw[Green] (18) to (19);\draw[blue] (20) to (17);
		%\draw[red, dashed] (1.5,1) to (4.25,5); \draw[red,dashed] (1.5,5)to (4,1);
		\draw (10.25,5) to (10.25,1);
		\draw (13,5) to (12.5,1); \draw (13,5) to (13,1); \draw (13,5) to (13.5,1);
		\draw(13)--(15); \draw (14)--(16);
		\draw[blue,,line width=.8mm] (1)--(3);\draw[blue,, line width=.8mm](6)--(8);\draw[blue,line width=.8mm] (12)--(16); \draw[blue, line width=.8mm] (15)--(17); \draw[blue, line width=.8mm] (20)--(22);
		\draw[Green,   line width=.8mm] (2)--(4); \draw[Green,  line width=.8mm] (5) to (7); \draw[Green,   line width=.8mm] (11)--(13); \draw[Green,   line width=.8mm] (14)--(18); \draw[Green,   line width=.8mm] (19)--(21); 
		
		\end{tikzpicture}
		
		\caption{The clean ladder $(H,U,W)$}
		\label{fig:graphHUW}
	\end{subfigure}
\caption{ }
\end{figure}
 
The red dashed lines in \cref{fig:maxrungseq} indicate the locations of rungs that cannot exist due to $(e^1,f^1)$ being a full cross.  
In \cref{fig:graphHUW}, the blue path represents the induced path $U$, the green path represents the induced path $W$, and the black lines represent rungs of ~$(H,U,W)$.
Notice that there is a rung in \cref{fig:maxrungseq} that has an endpoint in the $X$-span and an endpoint in the $Y$-span of~$(e^1,f^1)$, and this rung is not in~$(H,U,W)$. 

The next lemma follows from the process described above.    

\begin{lemma}\label{l:reroute-rail-induced}
	Let $\mathscr{X}$ be a maximal sequence of \indfull crosses of a messy ladder.
	Resolving the crosses of $\mathscr{X}$ results in a clean ladder.
\end{lemma}

In order to prove that a sufficiently large messy ladder conduces a clean ladder of the desired order, we will need the following lemmas to bound from above the order of the spans of crosses and the distance between consecutive \indfull crosses in a maximal sequence of \indfull crosses. 
We will combine these lemmas with \cref{l:reroute-rail-induced} to obtain a clean ladder of the desired order.  

The following definitions are essential to creating the bounds. 
Define a \emph{non-crossing matching} of a messy ladder to be a set of rungs that are pairwise non-adjacent and non-crossing. 
Let $M$ be a maximal non-crossing matching in $(L,X,Y)$.
The set of edges $M\cup \{\sigma,\tau\}$ is an \emph{augmented matching} of a messy ladder.
Note that the rungs of an augmented matching have the following properties: (1)~they are pairwise non-crossing; (2)~the only vertices that can be endpoints of at most two members of $M\cup\{\sigma,\tau\}$ are $\sigma_X$, $\sigma_Y$, $\tau_X$, and~$\tau_Y$.
A \emph{clean cycle} consists of two distinct rungs $e$ and $f$ that do not cross and the sub-paths of $X$ and $Y$ determined by $e_X$, $e_Y$, $f_X$, and $f_Y$.
A ladder is \emph{$r$-cycle-free} if it conduces no clean cycle of order $r$ or more.
A \emph{fan} of order $s$, denoted by $F_s$, is the graph obtained by taking an isolated vertex called the \emph{apex} and a path of order $s-1$ called the \emph{rim} and adding an edge between the apex and every vertex on the rim.
Let $\mathscr{F}_s$ be the family of graphs obtained from $F_s$ by subdividing each of the rim edges an arbitrary number, possibly zero, of times.
A member of the family $\mathscr{F}_s$ that is an induced subgraph of a ladder and has the apex on one rail and the rim entirely on the other rail is called \emph{clean}.
A ladder is said to be \emph{$s$-\ff} if it conduces no clean member of the family $\mathscr{F}_{s}$.

In the next three lemmas, we will consider messy ladders that are $r$-\cf\ and $s$-\ff\ for some values of $r$ and~$s$.
Our goal in those lemmas is to bound from above the order of the $X$-span and the $Y$-span of every cross by a function of $r$ and~$s$.
First, we bound from above the number of vertices on each of the rails of a messy ladder between two consecutive rungs in an augmented matching.

\begin{lemma}\label{lem:dist btw ncm rungs}
	Let $r$ and $s$ be integers exceeding three.  
	Let  $(L,X,Y)$ be a $r$-\cf\ and $s$-\ff\ messy ladder with an augmented matching~$M$.  
	There is an integer $f_{\ref{lem:dist btw ncm rungs}}(r,s)$ such that the number of vertices on each $X$ and $Y$ between a pair of consecutive rungs in $M$, including the endpoints of the rungs, is at most $f_{\ref{lem:dist btw ncm rungs}}(r,s)$. 
\end{lemma}

\begin{proof}
	Let $f_{\ref{lem:dist btw ncm rungs}}(r,s)=2((s-3)(r-4)+(s-2))+r-5$. 
	Our goal is to show that the number of vertices between a consecutive pair of augmented matching rungs is at most $f_{\ref{lem:dist btw ncm rungs}}(r,s)$.
	
	Let $(p^1,p^2,\dots, p^{\ell})$ be an augmented matching $M$ whose rungs are listed in the order of appearance on the rails.
	Let $X_j=X[p^j_X,p^{j+1}_X]$ and let $Y_j=~Y[p^j_Y,p^{j+1}_Y]$ for some $1\le j \le \ell-1$.
	Since $M$ is an augmented matching, it follows $L$ has no edge $xy$ such that $x\in X(p^j_X,p^{j+1}_X)$ and $y\in Y(p^j_Y,p^{j+1}_Y)$.
	Each of the vertices $p^j_X$, $p^j_Y$, $p^{j+1}_X$, and $p^{j+1}_Y$ has at most $s-2$ incident rungs, including the non-matching rungs, since $(L,X,Y)$ is $s$-\ff.
	
	We will bound from above the order of $Y_j$; the argument for $X_j$ is similar.
	Let $v_1$ be the vertex on $Y_j$ such that $p^j_Xv_1$ is a rung and the number of vertices on $Y[p^j_Y,v_1]$ is the maximum, and let $v_2$ be the vertex on $Y_j$ such that $p^{j+1}_Xv_2$ is a rung and the number of vertices on $Y[v_2,p_Y^{j+1}]$ is the maximum.
	We can express $Y[p^j_Y,p^{j+1}_Y]$ as $Y[p^j_Y, v_1]\cup Y(v_1,v_2)\cup  Y[v_2,p^{j+1}_U]$, where $Y(v_1,v_2)$ may be empty. 
	We first bound the order of $Y[p^j_Y,v_1]$. 
	Consider the vertices on $Y$ adjacent to~$p_X^j$.
	Since $(L,X,Y)$ is $r$-\cf\ and $s$-\ff, there are at most $s-2$ such vertices, and the sub-path of $Y$ between every two consecutive neighbors of $p_X^j$ has at most $r-4$ internal vertices.  
	There are at most $s-3$ of the sub-paths of $Y$ determined by the neighbors of $p_X^j$.
	
	So,  $Y[p^j_Y,v_1]$ has at most $s-2+(r-4)(s-3)$ vertices.
	Similarly, the number of vertices on $Y[v_2,p_Y^{j+1}]$ is at most~$(r-4)(s-3)+s-2$.	
	Since $(L,X,Y)$ is $r$-\cf, the number of vertices on $Y(v_1,v_2)$ is at most $r-5$.  
	So the number of vertices on $Y_j$ is at most $2((s-3)(r-4)+(s-2))+r-5$, as required.
\end{proof}

Next, we bound the number of rungs of an augmented matching that some other rung may cross.  

\begin{lemma}\label{lem:bound num ncm rungs cross}
	Let $r$ and $s$ be integers exceeding three.
	Let  $(L,X,Y)$ be a messy ladder that is $r$-\cf, $s$-\ff, and has an augmented matching~$M$.  
	There is an integer $f_{\ref{lem:bound num ncm rungs cross}}(r,s)$ such that if a rung crosses two rungs in $M$, then the number of vertices on the sub-paths of $X$ and $Y$ determined by endpoints of those two rungs in $M$ is at most~$f_{\ref{lem:bound num ncm rungs cross}}(r,s)$.  
\end{lemma}
\begin{proof}
	Let $f_{\ref{lem:bound num ncm rungs cross}}(r,s)=m_1m_2 +(m_1+1)(r-4)-1$ where 
	$m_1=(f_{\ref{lem:dist btw ncm rungs}}(r,s)-1)$, $m_2=(r-4)(s-3)+(s-2)$, and $f_{\ref{lem:dist btw ncm rungs}}(r,s)$ is the number from \cref{lem:dist btw ncm rungs}.
	Let $(p^1,p^2,\dots, p^{\ell})$ be an augmented matching $M$ whose rungs are listed in the order of appearance on the rails.
	Suppose $e$ is a rung such that $e_X$ is on $X[p_X^j,p_X^{j+1})$ and $e_Y$ is on $Y(p_Y^k,p_Y^{k+1}]$ for some $j$ and~$k$. 
	Without loss of generality, we may assume that~$j< k \le \ell-1$.
	Suppose that $e$ crosses two rungs $p^m$ and $p^n$ of~$M$. 
	Since $e$ crosses $p^m$ and $p^n$, it follows that $p^m_Y$ and $p^n_Y$ are on~$Y[p^{j+1}_Y,p^k_Y]$.
	We bound from above the number of vertices on the sub-path of $Y[p^{j+1}_Y,p^k_Y]$, as the argument for the number of vertices on $X$ is similar.
	
	\cref{lem:dist btw ncm rungs} implies the number of vertices on $X[e_X,p^{j+1}_X)$ is at most~$m_1=f_{\ref{lem:dist btw ncm rungs}}(r,s)-1$.
	
	Next, we bound from above the number of vertices on $Y[p^{j+1}_Y,e_Y]$. 
	For each vertex $v\in X[e_x,p^{j+1}_X)$, let $E_v$ be the set of rungs incident with the vertex~$v$.
	Let $D_v$ be the minimal sub-path of $Y$ that contains the endpoints of all the edges of~$E_v$.
	Since $(L,X,Y)$ is $s$-\ff, we have~$|E_v|<s-1$.
	Since $(L,X,Y)$ is $r$-\cf, there are at most $r-4$ internal vertices on $Y$ between every two consecutive rungs of $E_v$.
	This leads to the following inequality~$|V(D_v)|\le(r-4)(|E_v|-1)+|E_v|\le (r-4)(s-3)+(s-2)=m_2.$
	Let $\mathscr{D}= \bigcup_{v\in X[e_X,p^{j+1}_X)} D_v$.
	Since the union is taken over at most $m_1$ elements, the graph $\mathscr{D}$ has at most $m_1m_2$ vertices.  
	
	Let $\mathscr{Y}$ be the graph induced by the vertices of $Y[p^{j+1}_Y,e_Y]-\mathscr{D}$.
	Since there are at most $m_1$ components of $\mathscr{D}$, there are at most $m_1+1$ components of $\mathscr{Y}$.
	Since $(L,X,Y)$ is $r$-\cf, each component of $\mathscr{Y}$ has at most $r-4$ vertices, otherwise $G$ contains a cycle of order at least $r$.
	So, the number of vertices on $Y[p^{j+1}_Y,e_Y]$ is at most~$m_1m_2 +(m_1+1)(r-4)$.
	Every member of $M$ that crosses $e$ must have an endpoint on $Y[p^{j+1}_Y,e_Y)$.
	So the number of vertices on $Y[p^{j+1}_Y,p^k_Y]$ is at most $m_1m_2 +(m_1+1)(r-4)-1=f_{\ref{lem:bound num ncm rungs cross}}(r,s)$, as required.
\end{proof}

Now, we bound from above the number of vertices in each the $X$-span and the $Y$-span of a cross.

\begin{lemma}\label{lem:bounded spans}
	Let $r$ and $s$ be integers exceeding three. 
	There is an integer $f_{\ref{lem:bounded spans}}(r,s)$ such that if messy ladder $(L,X,Y)$ is $r$-\cf\ and $s$-\ff, then the $X$-span and the $Y$-span of every cross is bounded from above by~$f_{\ref{lem:bounded spans}}(r,s)$.
\end{lemma}
\begin{proof}
	Let $f_{\ref{lem:bounded spans}}(r,s)=2(f_{\ref{lem:bound num ncm rungs cross}}(r,s)+2f_{\ref{lem:dist btw ncm rungs}}(r,s)-2)$ where $f_{\ref{lem:dist btw ncm rungs}}(r,s)$ and $f_{\ref{lem:bound num ncm rungs cross}}(r,s)$ are the numbers from \cref{lem:dist btw ncm rungs,lem:bound num ncm rungs cross}, respectively.
	Let $(p^1,p^2,\dots, p^m)$ be an augmented matching $M$ whose rungs are listed in the order in which they appear on the rails.
	Note that $\sigma= p^1$ and $\tau=p^m$.
	For the cross $(e,f)$ suppose that $e_X\in ~X[p^j_X,p^{j+1}_X)$ and suppose that $f_X \in X(e_X,p^{\ell}_X]$ where $\ell$ is minimal subject to $f_X\le p^{\ell}_X$.
	Suppose that $e_Y\in Y(\sigma_Y,p^k_Y]$ where $k$ is minimal subject to $e_Y\le p_Y^k$ and suppose also that $f_Y\in Y[p^i_Y,e_Y)$ where~ $i$ is maximal subject to~$p^i_Y\le f_Y$.
	
	We will bound the number of vertices in the $Y$-span; the argument for the $X$-span is very similar.
	
	We will define $Y_1$ depending on the relation of $e_Y$ to $p^j_Y$.  
	If $p^j_Y\le e_Y$, then let $Y_1=Y[p^j_Y,p^{j+1}_Y]\cup Y(p^{j+1}_Y,p^{k-1}_Y)\cup Y[p^{k-1}_Y,e_Y]$. 
	By \cref{lem:bound num ncm rungs cross}, the sub-path $Y[p^{j+1}_Y,p^{k-1}_Y]$ has at most $f_{\ref{lem:bound num ncm rungs cross}}(r,s)$ vertices.
	So, the sub-path $Y(p^{j+1}_Y,p^{k-1}_Y)$ has at most $f_{\ref{lem:bound num ncm rungs cross}}(r,s)-2$ vertices.
	By \cref{lem:dist btw ncm rungs}, the number of vertices on $Y[p^j_Y,p^{j+1}_Y]$ is at most~$f_{\ref{lem:dist btw ncm rungs}}(r,s)$.
	Similarly, the number of vertices on $Y[p^{k-1}_Y,e_Y]$ is at most~$f_{\ref{lem:dist btw ncm rungs}}(r,s)$.
	So the number of vertices on $Y_1$ is at most~$2f_{\ref{lem:dist btw ncm rungs}}(r,s)+f_{\ref{lem:bound num ncm rungs cross}}(r,s)-2$. 
	If $e_Y<p^j_Y$, then let $Y_1=Y[e_Y,p^k_Y]\cup Y(p^k_Y,p^j_Y)\cup Y[p^j_Y,p^{j+1}_Y]$.
	By \cref{lem:bound num ncm rungs cross}, the sub-path $Y[p^k_Y,p^j_Y]$ has at most $f_{\ref{lem:bound num ncm rungs cross}}(r,s)$ vertices.
	So, the sub-path $Y(p^k_Y,p^j_Y)$ has at most $f_{\ref{lem:bound num ncm rungs cross}}(r,s)-2$ vertices.
	By \cref{lem:dist btw ncm rungs}, the number of vertices on $Y[p^j_Y,p^{j+1}_Y]$ is at most~$f_{\ref{lem:dist btw ncm rungs}}(r,s)$.
	Similarly, the number of vertices on $Y[p^{k-1}_Y,e_Y]$ is at most~$f_{\ref{lem:dist btw ncm rungs}}(r,s)$.
	So the number of vertices on $Y_1$ is at most~$2f_{\ref{lem:dist btw ncm rungs}}(r,s)+f_{\ref{lem:bound num ncm rungs cross}}(r,s)-2$. 
	
	Similarly, we will define $Y_2$ depending on the relation of $f_Y$ to $p^{\ell-1}_Y$. 
	If $p^{\ell-1}_Y\le f_Y$, then let $Y_2 =Y[p^{\ell-1}_Y,p^{\ell}_Y]\cup Y(p^{\ell}_Y,p^i_Y)\cup Y[p^i_Y,f_Y]$.  The argument for this case is analogous to the argument for $p^j_Y\le e_Y$.  
	If $f_Y< p^{\ell-1}_Y$, then let $Y_2= Y[f_Y,p^{i+1}_Y]\cup Y(p^{i+1}_Y, p^{\ell-1}_Y)\cup Y[p^{\ell-1}_Y,p^{\ell}_Y]$.
	Likewise, this case follows the argument when $e_Y<p^j_Y$; and thus, the number of vertices on $Y_2$ is at most~$2f_{\ref{lem:dist btw ncm rungs}}(r,s)+f_{\ref{lem:bound num ncm rungs cross}}(r,s)-2$.
	
	Combining the bounds on the number of vertices on $Y_1$ and $Y_2$, we get that $Y_1\cup Y_2$ has at most $2(f_{\ref{lem:bound num ncm rungs cross}}(r,s)+2f_{\ref{lem:dist btw ncm rungs}}(r,s)-2)$ vertices.  
	Since $e$ and $f$ cross, $Y[f_Y,e_Y]$ is a sub-path of $Y_1\cup Y_2$.
	Therefore, the number of vertices in the $Y$-span of a cross is at most $f_{\ref{lem:bounded spans}}(r,s)$, as required.
\end{proof}
Define a \emph{sub-ladder} $(L',X',Y')$ of $(L,X,Y)$ to be a messy ladder such that $X'$ and $Y'$ are rooted sub-paths of $X$ and $Y$, respectively, and $L'$ is the subgraph of $L$ induced by the vertices on~$X'\cup Y'$.
A \emph{cross-free ladder} is a messy ladder such that no pair of its rungs cross. 
A cross-free ladder is obviously a clean ladder. 
A messy ladder is \emph{$q$-cross-crowded} if it does not conduce a cross-free sub-ladder of order $q$ or more.
Note that a $q$-\cc~messy ladder is also $q$-\cf\ and~$q$-\ff.

In the previous lemmas, we considered messy ladders that were $r$-\cf\ and $s$-\ff\ for some integers $r$ and ~$s$.  
In the following lemma, we need a stronger assumption, namely that the messy ladder is $q$-\cc~for some integer $q$.
Since a cross-free ladder is a clean ladder, we restrict the order of the largest cross-free sub-ladder and show that a large $q$-\cc\ messy ladder
has a long maximal sequence of \indfull crosses.

\begin{lemma}\label{lem:manycrossesorlargegap}
	Let $q$ be an integer exceeding three and let $w$ be a positive integer.  
	There is an integer $f_{\ref{lem:manycrossesorlargegap}}(q,w)$ such that if a $q$-\cc~messy ladder has order at least $f_{\ref{lem:manycrossesorlargegap}}(q,w)$, then the length of every maximal sequence of \indfull crosses is at least~$w$. 
\end{lemma}
\begin{proof}
	Let $q$ and $w$ be integers such that $q \ge 4$ and $w \ge 1$,
	and let
	\[f_{\ref{lem:manycrossesorlargegap}}(q,w)=4(f_{\ref{lem:bounded spans}}(q,q)+1)(q^2+q)+2(w-1)f_{\ref{lem:bounded spans}}(q,q)+2(2f_{\ref{lem:bounded spans}}(q,q)+1)(q^2+q)(w-2)\] where $f_{\ref{lem:bounded spans}}(q,q)$ is the number from \cref{lem:bounded spans}.
	We prove that $f_{\ref{lem:manycrossesorlargegap}}(q,w)$ satisfies the conclusion.
	Suppose that $(L,X,Y)$ is a $q$-cross-crowded messy ladder that has a maximal sequence~$\mathscr{X}$ of pairwise independent crosses that has $z$ elements, where $0 \le z \le w-1$.
	We will show that the number of vertices of $L$ is less than $f_{\ref{lem:manycrossesorlargegap}}(q,w)$, thereby proving the lemma.
	
	If $z = 0$, then $(L,X,Y)$ is cross-free and so
	$|V(L)| < q \le f_{\ref{lem:manycrossesorlargegap}}(q,w)$,
	and the conclusion follows.
	So, for the remainder of the proof, we may assume that $z \ge 1$.
	Also, by symmetry, we may assume that $|V(X)| \ge |V(Y)|$, and concentrate on finding an upper bound only for $|V(X)|$.
	
	Let $\mathscr{X} = (\mathcal{X}_1, \mathcal{X}_2, \ldots, \mathcal{X}_z)$, and, for each $i$ in $\{1,2,\dots,z\}$, let $S_i$ be the $X$-span of $\mathcal{X}_i$.
	By \cref{lem:bounded spans}, the number of vertices of $S_i$ is at most	$f_{\ref{lem:bounded spans}}(q,q)$, and so the union $S$ of all $X$-spans of the crosses	in $\mathscr{X}$ has order at most $z f_{\ref{lem:bounded spans}}(q,q)$.
	
	Now, let $T = X - \bigcup\limits_{i=1}^z S_i$, and let $\mathcal{X}_i=(e^i,f^i)$ for each $i$ in $\{1,2,\dots,z\}$.
	Every connected component of $T$ is of one of the following forms:
	$X[\sigma_X,e^1_X)$, $X(e^z_X,\tau_X]$, and $X(f_X^i,e_X^{i+1})$ for some $i$ in $\{1, 2, \ldots,z-1\}$ in the case $z\ge 2$.
	Before finding upper bounds on the orders of such segments of $X$, we will present an upper bound
	on the number of rungs incident with vertices of those segments.
	
	First, we will bound the number of rungs incident with vertices of $X[\sigma_X,e_X^1)$. Note that the argument for the number of rungs incident with vertices on $X(f_X^z,\tau_X]$ is similar.
	Since each cross in $\mathscr{X}$ is full, each rung incident with a vertex in $X[\sigma_X,e_X^1)$ has the other endpoint on $Y[\sigma_Y,e_Y^1)$.
	Since $(L,X,Y)$ has no cross-free sub-ladder of order $q$, there are at most $q-1$ rungs with one endpoint on $X[\sigma_X,e_X^1)$ and other endpoint on $Y[\sigma_Y,f_Y^1]$.
	Since $(L,X,Y)$ is $q$-\ff, each vertex on $Y(f_Y^1,e_Y^1)$ is incident with at most $q-2$ rungs that have other endpoint on $X[\sigma_X,e_X^1)$.
	So there are at most $(q-2)(f_{\ref{lem:bounded spans}}(q,q)-2)$ rungs with one endpoint on $X[\sigma_X,e_X^1)$ and a distinct endpoint on $Y(f_Y^1,e_Y^1)$.
	Thus, there are at most $(q-2)(f_{\ref{lem:bounded spans}}(q,q)-2)+q-1$ rungs incident with a vertices of $X[\sigma_X,e_X^1)$.  
	Similarly, there are at most $(q-2)(f_{\ref{lem:bounded spans}}(q,q)-2)+q-1$ rungs incident with a vertices of $X(f_X^z,\tau_X]$.  
	
	Next, we assume that $z\ge 2$ and we bound from above the number of rungs incident with an arbitrary segment on $X$ between consecutive crosses of $\mathscr{X}$.
	Since each cross in $\mathscr{X}$ is full, each rung incident with a vertex in $X(f_X^i,e_X^{i+1})$ for some $i$ in $\{1,2\dots, z-1\}$ have the other endpoint on $Y(f_Y^i,e_Y^{i+1})$.
	Since $(L,X,Y)$ has no cross-free sub-ladder of order $q$, there are at most $q-1$ rungs with one endpoint on $X(f_X^i,e_X^{i+1})$ and other endpoint on $Y(e_Y^i,f_Y^{i+1})$.
	Since $(L,X,Y)$ is $q$-\ff, each vertex on one of $Y(f_Y^i,e_Y^i)$ and $Y(f_Y^{i+1},e_Y^{i+1})$ for some $i$ in $\{1,2,\dots, z-1\}$ is incident with at most $q-2$ rungs that have a distinct endpoint on $X(f_X^i,e_X^{i+1})$.
	So there are at most $2(q-2)(f_{\ref{lem:bounded spans}}(q,q)-2)$ rungs with one endpoint on $X(f_X^i,e_X^{i+1})$ and other endpoint on either $Y(f_Y^i,e_Y^i]$ or $Y[f_Y^{i+1},e_Y^{i+1})$.
	Thus, there are at most $2(q-2)(f_{\ref{lem:bounded spans}}(q,q)-2)+q-1$ rungs incident with a vertices of $X(f_X^i,e_X^{i+1})$.
	
	Now, we will bound from above the number of vertices on $X$.
	Since $(L,X,Y)$ is $q$-\cf, there are at most $q-4$ vertices between two rungs that are consecutive on $X$. 
	The bounds obtained above are cumbersome, and since this paper proves an existence result, we relax the number to obtain $f_{\ref{lem:manycrossesorlargegap}}(q,w)$. 
	The number of vertices on each of $X[\sigma_X,e_X^1)$ and $X(f_X^z,\tau_X]$ is at most \[(q-2)(f_{\ref{lem:bounded spans}}(q,q)-2)+q-1 + ((q-2)(f_{\ref{lem:bounded spans}}(q,q)-2)+q)(q-4)<(f_{\ref{lem:bounded spans}}(q,q)+1)(q^2+q).\]
	Similarly, the number of vertices on each $X(f_X^i,e_X^{i+1})$ is at most 
	\[(2(q-2)(f_{\ref{lem:bounded spans}}(q,q)-2)+q)(q-4)+2(q-2)(f_{\ref{lem:bounded spans}}(q,q)-2)+q-1<(2f_{\ref{lem:bounded spans}}(q,q)+1)(q^2+q).\]
	
	Since $z\le w-1$, the number of vertices on $X$ is at most, $g(q)=2((q-2)(f_{\ref{lem:bounded spans}}(q,q)-1)+q-1 + ((q-2)(f_{\ref{lem:bounded spans}}(q,q)-1)+q)(q-4))+(w-1)f_{\ref{lem:bounded spans}}(q,q) + (w-2)((2(q-2)(f_{\ref{lem:bounded spans}}(q,q)-1)+q)(q-4)+2(q-2)(f_{\ref{lem:bounded spans}}(q,q)-1)+q-1)$.
	Thus, the number of of vertices on $(L,X,Y)$ is at most $2g(q)<f_{\ref{lem:manycrossesorlargegap}}(q,w)$, a contradiction.
	Hence, every maximal sequence of pairwise independent full crosses has length at least $w$.	
\end{proof}	

The following lemma combines the previous lemmas in this section to complete the proof that a clean ladder of desired order is a sub-ladder of every messy ladder that is large enough.

\begin{lemma}\label{lem:messyfinite}
	Let $t$ be an integer exceeding two. 
	There is an integer $f_{\ref{lem:messyfinite}}(t)$ such that if a messy ladder with a maximal sequence of \indfull crosses $\mathscr{X}$ has order at least $f_{\ref{lem:messyfinite}}(t)$, then resolving the crosses of $\mathscr{X}$ results in a clean ladder of order at least~$t$.
\end{lemma}
\begin{proof}  
	Without loss of generality, we may assume that $t$ is even.
	We prove that $f_{\ref{lem:messyfinite}}(t)=f_{\ref{lem:manycrossesorlargegap}}(t,w)$, where $f_{\ref{lem:manycrossesorlargegap}}(t,w)$ is the number from \cref{lem:manycrossesorlargegap} and $w= \frac{t}{2}-1$, satisfies the conclusion.
	Suppose $(L,X,Y)$ is a messy ladder of order at least~$f_{\ref{lem:messyfinite}}(t)$ and $\mathscr{X}=(\X_1,\X_2,~\dots,~\X_z)$.
	
	If $(L,X,Y)$ conduces a cross-free sub-ladder of order at least $t$, then the conclusion holds.
	Therefore, we may assume that $(L,X,Y)$ is $t$-cross-crowded.
	
	By \cref{lem:manycrossesorlargegap}, it follows that $z\ge w$. 
	
	Let $(H,U,W)$ be the ladder obtained by resolving the crosses of $(L,X,Y)$.  
	By \cref{l:reroute-rail-induced}, $(H,U,W)$ is a clean ladder.
	It remains to show that $(H,U,W)$ has order at least~$t$.  
	By Remark~\ref{vertsofcl}, each full cross from the sequence has four vertices in $(H,U,W)$.  
	The cross $\X_1$ contributes four vertices to $(H,U,W)$ and each of the subsequent crosses contributes at least two new vertices to $(H,U,W)$.
	Thus $(H,U,W)$ has at least $4+2(w-1)=t$ vertices, as required.
\end{proof}

The following lemma combines the results from \cref{lem:longp,lem:messyfinite}.
We will use this lemma in the proof of the main result of the paper, \cref{thm:finite}.
\begin{lemma}\label{lem:messy-to-clean}
	Let $t$ be an integer exceeding three.  
	There is an integer $f_{\ref{lem:messy-to-clean}}(t)$ such that if a messy ladder $(L,X,Y)$ has a path of order at least $f_{\ref{lem:messy-to-clean}}(t)$, then $(L,X,Y)$ conduces one of the following: $K_t$, $K_{2,t}$, $K_{2,t}^+$, and a clean ladder of order at least~$t$.  	
\end{lemma}
\begin{proof}
	Let $f_{\ref{lem:messy-to-clean}}(t)=f_{\ref{lem:longp}}(t,q)$ where ~$q=f_{\ref{lem:messyfinite}}(t)$. 
	Since $G$ has a path of order at least~$f_{\ref{lem:longp}}(t,q)$, \cref{lem:longp} asserts that $G$ conduces one of the following: $K_t$, $K_{2,t}$, $K_{2,t}^+$, and a messy ladder of order at least~$q$. 		
	If $G$ conduces a messy ladder of order at least $q$, then \cref{lem:messyfinite} asserts that $G$ conduces a clean ladder of order at least~$t$, as required.  
\end{proof}

\section{Proving Main \cref{thm:finite}}\label{Proving Main Theorem}
The main theorem is a straightforward consequence of \cref{lem:shortp,lem:messy-to-clean}.
Specifically, for some integer $r$ exceeding two, if a 2-connected graph $G$ has a sufficiently long path, then it conduces one of the following: $K_r$, $K_{2,r}$, $K_{2,r}^+$, and a clean ladder of order at least $r$; and if $G$ fails to have a sufficiently long path, but is large enough, then $G$ conduces a member of one of the families $\mathscr{K}_{2,r}^+$ and ~$\mathscr{K}_{2,r}$.

\begin{namedthm*}{\cref{thm:finite}}
	Let $r$ be an integer exceeding two. 
	There is an integer $f_{\ref{thm:finite}}(r)$ such that every $2$-connected graph of order at least $f_{\ref{thm:finite}}(r)$ conduces one of the following: $K_r$, a clean ladder of order at least $r$, a member of $\mathscr{K}_{2,r}$, and a member of $\mathscr{K}_{2,r}^+$.
\end{namedthm*}
\begin{proof}
	%Combine \cref{lem:shortp}, \cref{lem:longp}, and \cref{lem:messyfinite}. 
	Let $f_{\ref{thm:finite}}(r)=f_{\ref{lem:shortp}}(q,r)$ where ~$q=f_{\ref{lem:messy-to-clean}}(r)$. 
	Since $G$ has at least $f_{\ref{lem:shortp}}(q,r)$ vertices, \cref{lem:shortp} asserts that $G$ has a path of order $q$ or conduces a member of one of the families  $\mathscr{K}_{2,r}$ and $\mathscr{K}_{2,r}^{+}$.  
	If $G$ conduces a path of order $q$, then, by \cref{lem:messy-to-clean}, the graph $G$ conduces one of the following: $K_r$, $K_{2,r}$, $K_{2,r}^+$, and a clean ladder of order at least~$r$.
	This completes the proof.
\end{proof}

\bibliographystyle{plain}
\bibliography{bibliography}{}
\vspace{-5mm}

\end{document}